\documentclass[11pt, reqno]{amsart}

\usepackage{amsmath,mathrsfs}
\usepackage{amsfonts}
\usepackage{amssymb}
\usepackage{graphicx}
\usepackage[square,numbers]{natbib}

\newtheorem{theorem}{Theorem}
\newtheorem{claim}[theorem]{Claim}

\newtheorem{lemma}[theorem]{Lemma}
\newtheorem{proposition}[theorem]{Proposition}

\theoremstyle{definition}
\newtheorem{definition}[theorem]{Definition}

\theoremstyle{remark}

\newtheorem{case}{Case}

\numberwithin{theorem}{section}
\numberwithin{equation}{section}

\def\XXint#1#2#3{{\setbox0=\hbox{$#1{#2#3}{\int}$}
     \vcenter{\hbox{$#2#3$}}\kern-.5\wd0}}

\newcommand{\dd}{\; \mathrm{d}}

\begin{document}
\title[Lusin approximation in step 2 Carnot groups]{Lusin approximation for horizontal curves in step 2 Carnot groups}
\author[Enrico Le Donne]{Enrico Le Donne}
\address[Enrico Le Donne]{University of Jyvaskyla, Finland}
\email{Enrico.E.LeDonne@jyu.fi}
\author[Gareth Speight]{Gareth Speight}
\address[Gareth Speight]{University of Cincinnati, USA}
\email{Gareth.Speight@uc.edu}
\date{February 8, 2016}

\renewcommand{\subjclassname} {\textup{2010} Mathematics Subject Classification}


\begin{abstract}
A Carnot group $\mathbb{G}$ admits Lusin approximation for horizontal curves if for any absolutely continuous horizontal curve $\gamma$ in $\mathbb{G}$ and $\varepsilon>0$, there is a $C^1$ horizontal curve $\Gamma$ such that $\Gamma=\gamma$ and $\Gamma'=\gamma'$ outside a set of measure at most $\varepsilon$. We verify this property for free Carnot groups of step 2 and show that it is preserved by images of Lie group homomorphisms preserving the horizontal layer. Consequently, all step 2 Carnot groups admit Lusin approximation for horizontal curves.
\end{abstract}
\maketitle

\section{Introduction}
Carnot groups (Definition \ref{Carnot}) are Lie groups whose Lie algebra admits a stratification. The stratification gives rise to dilations and implies that points can be connected by horizontal curves (Definition \ref{horizontalcurve}), which are absolutely continuous curves with tangents in a distinguished subbundle of the tangent bundle \cite{Cho39, Ras38}. The Carnot-Carath\'{e}odory distance is the distance associated to a left-invariant length structure. Moreover, on each Carnot group one naturally considers Haar measures, which are unique up to scaling. With so much structure, it has become highly interesting to study geometric measure theory, and other aspects of analysis or geometry, in Carnot groups  \cite{ABB, B94, BLU07, CDPT07, Gro96, LeD, Mon02, Vit14}.

The setting of Carnot groups has both similarities and differences compared to the Euclidean case. On the one hand, there is often enough structure to ask if a result in the Euclidean setting generalizes to Carnot groups. Sometimes this is the case - for example, one can prove a meaningful version of Rademacher's theorem for Lipschitz maps between Carnot groups (Pansu's Theorem \cite{Pan89, Mag01}). On the other hand, Carnot groups exhibit fractal behaviour, the reason being that on such metric spaces the only curves of finite length are the horizontal ones. Moreover, any Carnot group (except for Euclidean spaces themselves) contains no subset of positive measure that is bi-Lipschitz equivalent to a subset of a Euclidean space \cite{LeD, Sem96}.

Absolutely continuous curves are those for which displacements are given by integrating the derivative - this connection between position and velocity makes them ideal for geometry. It is well known that absolutely continuous curves in $\mathbb{R}^{n}$ admit a Lusin approximation by $C^{1}$ curves \cite{AT04, EG91}: if $\gamma\colon [a,b]\to \mathbb{R}^{n}$ is absolutely continuous and $\varepsilon>0$, then there exists a $C^1$ curve $\Gamma\colon [a,b]\to \mathbb{R}^{n}$ such that:
\[\mathcal{L}^1\{t\in [a,b]:\Gamma(t)\neq \gamma(t) \mbox{ or }\Gamma'(t)\neq \gamma'(t)\}<\varepsilon.\]

Since only horizontal curves are important in the geometry of Carnot groups, one should ask if any horizontal curve in a Carnot group can be approximated by a horizontal curve that is also $C^{1}$. 

\begin{definition}
We say that a Carnot group $\mathbb{G}$ \emph{admits Lusin approximation for horizontal curves} if the following is true. Suppose $\gamma\colon [0,1]\to \mathbb{G}$ is an absolutely continuous horizontal curve in $\mathbb{G}$ and $\varepsilon>0$. Then there exists a $C^1$ horizontal curve $\Gamma\colon [0,1]\to \mathbb{G}$ such that:
\[\mathcal{L}^{1} \{t\in [0,1] \colon \Gamma(t)\neq \gamma(t) \mbox{ or }\Gamma'(t)\neq \gamma'(t)\}<\varepsilon.\]
\end{definition}

In \cite{Spe14} the second-named author showed that this is the case for the Heisenberg group, the most frequently studied non-Euclidean Carnot group (which has step 2), but not true for the Engel group (which has step 3). Roughly, the step of a Carnot group (Definition \ref{Carnot}) is the number of commutators one needs to take before the horizontal directions generate the whole tangent space. Such a value coincides with the nilpotency step of the group. Thus in Carnot groups of higher step, more complicated horizontal curves are needed to connect points. Euclidean spaces correspond to step 1 Carnot groups, where all absolutely continuous curves are horizontal.

In this article we extend the positive result of \cite{Spe14} to all Carnot groups of step 2: any horizontal curve in a step-2 Carnot group can be approximated by a $C^{1}$ horizontal curve.

\begin{theorem}\label{lusinstep2}
All Carnot groups of step 2 admit Lusin approximation for horizontal curves.
\end{theorem}

To achieve the above result we first prove the desired claim in free Carnot groups of step 2 (Definition~\ref{freegroupstep2} and Theorem~\ref{freelusin}). To deduce Theorem~\ref{lusinstep2} from Theorem~\ref{freelusin}, we first show that if a Carnot group supports the Lusin approximation for horizontal curves, then so does any image of that space under a Lie group homomorphism that preserves the horizontal layer (Theorem~\ref{imagespreserve}). We then observe that any Carnot group is a suitable image of a free Carnot group of the same step (Theorem \ref{goodmapsexist}). 

For each integer $r\geq 2$ (corresponding to the dimension of the horizontal bundle), there is a free Carnot group of rank $r$ and step 2. Informally this is the Lie group associated to the Lie algebra of rank $r$ and step 2 for which commutators of horizontal vectors satisfy the least number of relations allowable. However, each free Carnot group of rank $r$ and step 2 admits a concrete description in coordinates (Definition \ref{freegroupstep2}). This description is more complicated than that of the Heisenberg group but has some aspects in common. Most importantly: horizontal curves are given by lifting curves in a Euclidean space, and the different vertical coordinates of such a lift (corresponding to non-horizontal directions) are given by signed areas swept out by the curve in various horizontal planes. 

The proof of Theorem \ref{freelusin}, namely Lusin approximation in free Carnot groups of step 2, forms the main part of the paper. We first restrict to a compact set $K$ of large measure such that $\gamma'|_{K}$ is uniformly continuous and each point of $K$ is a Lebesgue point of $\gamma'$. The complement of such a compact set is a union of intervals $(a,b)$. Given such an interval, we find an interpolating $C^{1}$ horizontal curve $\Gamma$ on $[a,b]$, with $\Gamma(t)=\gamma(t)$ and $\Gamma'(t)=\gamma'(t)$ for $t=a$ and $t=b$, such that $\Gamma'(t)$ is close to $\gamma'(a)$ and $\gamma'(b)$ for all $t\in (a,b)$. We do this by using our choice of $K$ to give information about the boundary conditions, then using the above interpretation of horizontal curves and constructing curves in Euclidean spaces satisfying boundary conditions and sweeping out given signed areas in different planes. 

The main difference between free Carnot groups and the Heisenberg group as treated in \cite{Spe14} is that in the Heisenberg group there is only one vertical (i.e. non-horizontal) coordinate, whereas for free Carnot groups there can be many. Hence the control on boundary conditions and construction of the interpolating horizontal curve is more complicated. Such a curve is constructed as a concatenation of simpler curves, each of which trace of a desired signed area in one plane but does not affect signed areas previously controlled in other planes.

Validity of a Lusin $C^{1}$ approximation is related to existence of a Whitney extension theorem. Working independently, Zimmermann \cite{Zim15} recently showed that a Whitney extension theorem holds for horizontal curve (fragments) in the Heisenberg group. He also observed that this can be used to give another proof of the Lusin approximation in the Heisenberg group as appeared in \cite{Spe14}. Whitney extension and Lusin approximation in Carnot groups have also been considered by other authors, but for real-valued functions defined on Carnot groups \cite{FSS01, VP06}. In the Euclidean case, Lusin approximation and Whitney extension can be used to show that rectifiable sets can be defined equivalently using both Lipschitz surfaces and $C^1$ surfaces. Rectifiability is an active area of research in metric spaces and particularly Carnot groups \cite{AK00, BRS03, FSS01, FSS03, FSS03b, GJ13, KS04, Mag04, MSS10, P04}.

\medskip

\noindent \textbf{Acknowledgement.} Part of this work was carried out when the second-named author was visiting the University of Jyvaskyla. He thanks the mathematics department at the University of Jyvaskyla for its kind hospitality. E.L.D. acknowledges the support of the Academy of Finland project no. 288501.   

\section{Preliminaries}
In this section we define Carnot groups, horizontal curves, free-nilpotent Lie algebras and free Carnot groups, and coordinate representations of free Carnot groups of step 2.

\subsection{Carnot groups}
Recall that a Lie group is a smooth manifold that is also a group for which the multiplication and inversion are smooth. A smooth vector field $X$ on a Lie group $\mathbb{G}$ is left-invariant if it satisfies the equality $(dL_{g})(X(h))=X(gh)$ for all $g, h\in \mathbb{G}$, where $dL_{g}$ denotes the differential of the left translation $h\mapsto gh$. The associated Lie algebra of a Lie group is the vector space of left-invariant vector fields equipped with the Lie bracket operation on vector fields. Using left translations, this can be identified with the tangent space at the identity. Viewing smooth vector fields as derivations on smooth functions, the Lie bracket $[X,Y]$ of smooth vector fields $X$ and $Y$ is a smooth vector field defined on smooth functions by $[X,Y](f)=X(Y(f))-Y(X(f))$.

\begin{definition}\label{Carnot}
A \emph{Carnot group} $\mathbb{G}$ of \emph{step} $s$ is a simply connected Lie group whose Lie algebra $\mathfrak{g}$ admits a decomposition as a direct sum of subspaces of the form
\[\mathfrak{g}=V_{1}\oplus V_{2}\oplus \ldots \oplus V_{s}\]
such that $V_{i}=[V_{1},V_{i-1}]$ for any $i=2, \ldots, s$, and $[V_{1},V_{s}]=0$. The subspace $V_{1}$ is called the \emph{horizontal layer} and its elements are called \emph{horizontal vector fields}. 
\end{definition}

Recall that a curve $\gamma\colon [a,b]\to \mathbb{R}^{n}$ is absolutely continuous if it is differentiable almost everywhere, $\gamma'\in L^{1}[a,b]$, and
\[\gamma(t_{2})=\gamma(t_{1})+\int_{t_{1}}^{t_{2}} \gamma'(t) \dd t\]
whenever $t_{1}, t_{2}\in [a,b]$. The family of absolutely continuous curves in $\mathbb{R}^{n}$ includes all Lipschitz curves. A curve in a smooth manifold is absolutely continuous if in any chart it is absolutely continuous.

\begin{definition}\label{horizontalcurve}
Fix a basis $X_{1}, \ldots, X_{r}$ of $V_{1}$. An absolutely continuous curve $\gamma\colon [a,b]\to \mathbb{G}$ is \emph{horizontal} if there exist $u_{1}, \ldots, u_{r}\in L^{1}[a,b]$ such that
\[\gamma'(t)=\sum_{j=1}^{r}u_{j}X_{j}(\gamma(t))\]
for almost every $t\in [a,b]$.
\end{definition}

We say that a vector $v\in \mathbb{R}^n$ is \emph{horizontal} at $p\in \mathbb{G}$ if $v=E(p)$ for some $E\in V_1$. Thus an absolutely continuous curve is horizontal if and only if $\gamma'(t)$ is horizontal at $\gamma(t)$ for almost every $t$. Left translations map absolutely continuous, $C^{1}$ or horizontal curves to curves of the same type. 

\subsection{Free Carnot groups of step 2}

We first give the more abstract but flexible definition of free Carnot group via free-nilpotent Lie algebras. We then give the representation in coordinates of free Carnot groups of step 2 which we will use for most of the article. 

We recall that a Lie algebra is a vector space $V$ equipped with a Lie bracket $[\cdot,\cdot] \colon V\times V\to V$ that is bilinear, satisfies $[x,x]=0$, and satisfies the Jacobi identity $[x,[y,z]]+[z,[x,y]]+[y,[z,x]]=0$. A homomorphism between Lie algebras is simply a linear map that preserves the Lie bracket. If it is also bijective than the map is an isomorphism. Free-nilpotent Lie algebras can be defined as follows (see Definition 14.1.1 in \cite{BLU07}).

\begin{definition}\label{freeliealgebra}
Let $r\geq 2$ and $s\geq 1$ be integers. We say that $\mathcal{F}_{r,s}$ is the \emph{free-nilpotent Lie algebra} with $r$ \emph{generators} $x_1, \ldots, x_{r}$ of \emph{step} $s$ if:
\begin{enumerate}
\item $\mathcal{F}_{r,s}$ is a Lie algebra generated by elements $x_1, \ldots, x_r$ (i.e., $\mathcal{F}_{r,s}$ is the smallest Lie algebra containing $x_1, \ldots, x_r$),
\item $\mathcal{F}_{r,s}$ is nilpotent of step $s$ (i.e., nested Lie brackets of length $s+1$ are always $0$),
\item for every Lie algebra $\mathfrak{g}$ that is nilpotent of step $s$ and for every map $\Phi\colon \{x_1, \ldots, x_r\}\to \mathfrak{g}$, there exists a homomorphism of Lie algebras $\tilde{\Phi}\colon \mathcal{F}_{r,s} \to \mathfrak{g}$ that extends $\Phi$, and moreover it is unique.
\end{enumerate}
\end{definition}

We next define free Carnot groups (see Definition 14.1.3 in \cite{BLU07}). A homomorphism between Carnot groups is simply a smooth map that preserves the group operation and the horizontal vectors. If the map is also a diffeomorphism then it is an isomorphism. Since Carnot groups are simply connected Lie groups, any homomorphism $\phi$ between their Lie algebras lifts to a homomorphism $F$ between the Carnot groups satisfying $dF=\phi$. Note that if two Carnot groups have isomorphic Lie algebras, then the Carnot groups themselves are isomorphic.

\begin{definition}\label{freecarnotgroup}
A \emph{free Carnot group} is a Carnot group whose Lie algebra is isomorphic to a free-nilpotent Lie algebra $\mathcal{F}_{r, s}$ for some $r\geq 2$ and $s\geq 1$. In this case the horizontal layer of the free Carnot group is the span of the generators of $\mathcal{F}_{r,s}$.
\end{definition}

We now give an explicit construction of free Carnot groups of step 2. Fix an integer $r\geq 2$ and denote $n=r+r(r-1)/2$. In $\mathbb{R}^{n}$ denote the coordinates by $x_{i}$, $1\leq i\leq r$, and $x_{ij}$, $1\leq j<i\leq r$. Let $\partial_{i}$ and $\partial_{ij}$ denote the standard basis vectors in this coordinate system. Define $n$ vector fields on $\mathbb{R}^n$ by:
\[ X_{k}:=\partial_{k}+\sum_{j>k} \frac{x_{j}}{2}\partial_{jk}-\sum_{j<k}\frac{x_{j}}{2}\partial_{kj} \qquad \mbox{if }1\leq k \leq r,\]
\[X_{kj}:=\partial_{kj} \qquad \mbox{if } 1\leq j<k\leq r.\]

\begin{definition}\label{freegroupstep2}
Define the \emph{free Carnot group of step 2 and $r$ generators} by $\mathbb{G}_r:=(\mathbb{R}^{n}, \, \cdot)$, where the product $x\cdot y \in \mathbb{G}_r$ of $x,y\in \mathbb{G}_r$ is given by:
\[(x\cdot y)_{k} = x_{k}+y_{k} \qquad \mbox{if }1\leq k\leq r,\]
\[(x\cdot y)_{ij} = x_{ij}+y_{ij}+\frac{1}{2}(x_{i}y_{j}-y_{i}x_{j}) \qquad \mbox{if }1\leq j<i\leq r.\]
The Carnot structure of $\mathbb{G}_r$ is given by
\[V_{1}=\mathrm{Span} \{X_{k}\colon 1\leq k \leq r\} \mbox{ and }V_{2}=\mathrm{Span} \{X_{kj}\colon 1\leq j<k\leq r\}.\]
\end{definition}

Note that free Carnot groups of step 2 are exactly those that are isomorphic to a Carnot group $\mathbb{G}_r$ for some $r$ and $X_{1}, \ldots, X_{r}$ are the generators.

It is easily verified that for $1\leq j<k \leq r$ and $1\leq i\leq r$:
\[ [X_{k}, X_{j}]=X_{kj} \mbox{ and } [X_{i}, X_{kj}]=0.\]

It will be important that horizontal curves in $\mathbb{G}_r$ are obtained by lifting curves in $\mathbb{R}^r$. The following proposition follows directly from the definitions of $X_{1}, \ldots, X_{r}$, so the proof is omitted.

\begin{proposition}\label{horizontalequation}
A vector $v\in \mathbb{R}^{n}$ is horizontal at $p\in \mathbb{G}_r$ if and only if for every $1\leq j<i\leq r$:
\[v_{ij}=\frac{1}{2}(p_{i}v_{j}-p_{j}v_{i}).\]
An absolutely continuous curve $\gamma \colon [a,b]\to \mathbb{G}_r$ is horizontal if and only if for every $1\leq j<i\leq r$:
\[\gamma_{ij}(t)=\gamma_{ij}(a) + \frac{1}{2}\int_{a}^{t} (\gamma_{i}\gamma_{j}'-\gamma_{j}\gamma_{i}')\]
for every $t\in [a,b]$.
\end{proposition}

If $(\gamma_{i}(a),\gamma_{j}(a))=(0,0)$ then $\frac{1}{2}\int_{a}^{t} (\gamma_{i}\gamma_{j}'-\gamma_{j}\gamma_{i}')$ can be interpreted as the signed area of the planar region enclosed by the curve $(\gamma_{i},\gamma_{j})|_{[a,t]}$ and the straight line segment joining $(0,0)$ to $(\gamma_{i}(t),\gamma_{j}(t))$.

\begin{definition}\label{horizontallift}
Suppose $\varphi\colon [a,b] \to \mathbb{R}^{r}$ is absolutely continuous, $p\in \mathbb{G}_r$ and $\varphi(a)=(p_{1},\ldots, p_{r})$. Define $\gamma\colon [a,b] \to \mathbb{G}_r$ by $\gamma_{i}=\varphi_{i}$ for $1\leq i \leq r$ and, for every $1\leq j<i\leq r$,
\[\gamma_{ij}(t)=p_{ij} + \frac{1}{2} \int_{a}^{t} (\varphi_{i}\varphi_{j}'-\varphi_{j}\varphi_{i}')\]
for every $t\in [a,b]$. We say that $\gamma$ is the \emph{horizontal lift} of $\varphi$ starting at $p$.
\end{definition}

Proposition \ref{horizontalequation} shows that the horizontal lift of an absolutely continuous curve in $\mathbb{R}^r$ is a horizontal curve in $\mathbb{G}_r$. 

Throughout this paper, $C$ will denote a bounded constant which may vary between expressions.

\section{Lusin approximation in free Carnot groups of step 2}\label{concreteargument}

The main result of this section is stated below. It will yield our main result, Theorem \ref{lusinstep2}, when combined with Theorem \ref{imagespreserve} and Theorem \ref{goodmapsexist}.

\begin{theorem}\label{freelusin}
Free Carnot groups of step 2 admit Lusin approximation for horizontal curves.
\end{theorem}

Fix an integer $r\geq 2$, $n=r+r(r-1)/2$ and let $\mathbb{G}=\mathbb{G}_r$ be the free Carnot group of step 2 and rank $r$ (Definition \ref{freegroupstep2}). Fix an absolutely continuous horizontal curve $\gamma\colon [0,1]\to \mathbb{G}$ and $\varepsilon>0$.

Our proof of Theorem \ref{freelusin} has three parts. We first define a compact set $K$ of large measure on which $\gamma'$ is well behaved and, by composing with an isomorphism of $\mathbb{G}$, investigate the properties of $\gamma$ at nearby points in a good choice of coordinates. We then construct interpolating curves, which are smooth and horizontal, in the subintervals that form the complement of $K$. Finally we piece together our curves to obtain the desired smooth horizontal curve, which agrees with $\gamma$ on $K$.

\subsection{A good compact set}
Since $\gamma$ is horizontal (hence absolutely continuous), $\gamma'$ is integrable so almost every $t_{0}\in [0,1]$ is a Lebesgue point of $\gamma'$. Using also Lusin's theorem, we may choose a compact set $K\subset (0,1)$ and a function $\delta\colon (0,\infty) \to (0,\infty)$ such that $\mathcal{L}^{1}([0,1]\setminus K)<\varepsilon$ and:
\begin{equation}\label{l1} \gamma'(t_{0}) \mbox{ exists and is horizontal at }\gamma(t_0) \mbox{ for every }t_0 \in K,\end{equation}
\begin{equation}\label{l2} |\gamma'(t_{1})-\gamma'(t_{0})|\leq \eta \mbox{ if }t_{0}, t_{1} \in K \mbox{ and }|t_{1}-t_{0}|<\delta(\eta),\end{equation}
\begin{equation}\label{l3} \int_{t_0}^{t_{0}+r} |\gamma' - \gamma'(t_{0})|\leq \eta r \mbox{ if } t_0 \in K \mbox{ and } 0<r<\delta(\eta).\end{equation}

Let $m=\min K$ and $M=\max K$. To prove Theorem \ref{freelusin}, it suffices to find a $C^1$ horizontal curve $\Gamma\colon [m,M]\to \mathbb{G}$ with $\Gamma=\gamma$ and $\Gamma'=\gamma'$ on $K$. After this is done one can simply extend $\Gamma$ arbitrarily to a $C^1$ horizontal curve on $[0,1]$. 

Fix $a\in K$. As $\gamma'(a)$ is horizontal at $\gamma(a)$, we may write $\gamma'(a)=E(\gamma(a))$ for $E\in V_{1}$. Choose a linear bijection $\Phi_{a} \colon V_{1} \to V_{1}$ such that $\Phi_{a}(E)=LX_{1}$ for some $L\geq 0$ and preserving the inner product induced by the basis $X_{1}, \ldots, X_{r}$ of $V_{1}$. Since $\mathbb{G}$ is free-nilpotent, $\Phi_{a}$ extends to an isomorphism of the Lie algebra of $\mathbb{G}$ (Lemma 14.1.4 \cite{BLU07}). Since Carnot groups are simply connected Lie groups, such an isomorphism lifts to a Lie group isomorphism $F_{a}\colon \mathbb{G}\to \mathbb{G}$ satisfying $dF_{a}|_{V_{1}}=\Phi_{a}$. The map $F_{a}$ is smooth and $dF_{a}$ preserves $V_{1}$ and $V_{2}$, in particular $F_{a}$ sends horizontal/smooth curves to horizontal/smooth curves. The definition of $F_{a}$ and of $X_{1}, \ldots, X_{r}$ imply that $F_{a}$ has the form $F_{a}(x,y)=(A(x),B(y))$, where $A\colon \mathbb{R}^{r} \to \mathbb{R}^{r}$ is a linear isometry and $B\colon \mathbb{R}^{n-r}\to \mathbb{R}^{n-r}$ is linear.

Recall that $L_{g}\colon \mathbb{G}\to \mathbb{G}$ denotes the left translation $h\mapsto gh$. Define $\varphi\colon [0,1]\to\mathbb{G}$ by:
\[\varphi=F_{a}\circ L_{\gamma(a)^{-1}} \circ \gamma.\]
Clearly $\varphi$ depends heavily on $a$, though to keep the expressions simpler we don't emphasize this in our choice of notation. Notice:
\[\varphi(a)=F_{a}(\gamma(a)^{-1}\gamma(a))=F_{a}(0)=0,\]
since group homomorphisms preserve the identity. Further:
\begin{align*}
\varphi'(a)=dF_{a}\circ dL_{\gamma(a)^{-1}}(E(\gamma(a)))&=dF_{a}(E(0))\\
&=LX_{1}(0)\\
&=(L,0,\ldots,0),
\end{align*}
using left-invariance of $E$ and our choice of $F_{a}$. Since $F_{a}$ and $L_{\gamma(a)^{-1}}$ are smooth with differentials preserving $V_{1}$, \eqref{l1} implies that:
\begin{equation}\label{lc1} \varphi'(t_{0}) \mbox{ exists and is horizontal at }\varphi(t_0)\mbox{ for every }t_0 \in K.\end{equation}
Since the group translation is Euclidean in the first $r$ coordinates, $F_{a}$ acts as a linear isometry when restricted to the first $r$ coordinates. Since $\gamma'$ is continuous on the compact set $K$, we note:
\begin{equation}\label{Lbounded} L \mbox{ is uniformly bounded from above independently of }a\in K. \end{equation}
Further, \eqref{l2} and \eqref{l3} respectively imply that for $1\leq i\leq r$:
\begin{equation}\label{lc2} |\varphi_{i}'(t_{1})-\varphi_{i}'(t_{0})|\leq \eta \mbox{ if }t_{0}, t_{1} \in K \mbox{ and } |t_{1}-t_{0}|<\delta(\eta),\end{equation}
\begin{equation}\label{lc3} \int_{t_0}^{t_{0}+r} |\varphi_{i}' - \varphi_{i}'(t_{0})| \leq \eta r \mbox{ if }t_0 \in K \mbox{ and } 0<r<\delta(\eta).\end{equation}

\begin{lemma}\label{curvecontrol}
Let $a\in K$ and denote $\varphi=F_{a}\circ L_{\gamma(a)^{-1}} \circ \gamma$ as described above. For any $\eta>0$,  $a<t<a+\delta(\eta)$ implies:
\begin{enumerate}
\item $|\varphi_{i}(t)-(t-a)\varphi_{i}'(a)|\leq \eta(t-a)$ for $1\leq i\leq r$,
\item $|\varphi_{ij}(t)|\leq \eta( |\varphi_{i}'(a)| + |\varphi_{j}'(a)| +\eta)(t-a)^{2}$ for $1\leq j<i\leq r$.
\end{enumerate}
\end{lemma}

\begin{proof}
Using $\varphi(a)=0$ and \eqref{lc3} gives
\[|\varphi_{i}(t)-(t-a)\varphi_{i}'(a)|\leq \int_{a}^{t} |\varphi_{i}'-\varphi_{i}'(a)| \leq \eta(t-a)\]
for any $a<t<a+\delta(\eta)$ and $1\leq i\leq r$. This proves (1).

Fix $a<t<a+\delta(\eta)$ and $i>j$. Since $\varphi$ is a horizontal curve and $\varphi(a)=0$, Proposition \ref{horizontalequation} gives:
\[\varphi_{ij}(t)=\frac{1}{2} \int_{a}^{t}(\varphi_{i} \varphi_{j}' - \varphi_{j}\varphi_{i}').\]
We estimate as follows:
\begin{align*}
& \int_{a}^{t} |\varphi_{i}(s)\varphi_{j}'(s)-(s-a)\varphi_{i}'(a)\varphi_{j}'(a)| \dd s\\
&\qquad \leq \int_{a}^{t} \Big(|\varphi_{i}(s)\varphi_{j}'(s)-(s-a)\varphi_{i}'(a)\varphi_{j}'(s)|\\
&\qquad \qquad \qquad \qquad +|(s-a)\varphi_{i}'(a)\varphi_{j}'(s)-(s-a)\varphi_{i}'(a)\varphi_{j}'(a)|\Big) \dd s\\
&\qquad \leq \int_{a}^{t} |\varphi_{j}'(s)||\varphi_{i}(s)-(s-a)\varphi_{i}'(a)|+(s-a)|\varphi_{i}'(a)||\varphi_{j}'(s)-\varphi_{j}'(a)|\dd s.
\end{align*}
For the first term, we use the fact that \eqref{lc3} implies
\[ \int_{a}^{t} |\varphi_{j}'(s)-\varphi_{j}'(a)| \dd s\leq \eta (t-a)\]
and (1) implies $|\varphi_{i}(s)-(s-a)\varphi_{i}'(a)|\leq \eta (s-a)$ for $s\in (a,t)$. Hence:
\begin{align*}
\int_{a}^{t} |\varphi_{j}'(s)||\varphi_{i}(s)-(s-a)\varphi_{i}'(a)| \dd s &\leq \eta \int_{a}^{t} |\varphi_{j}'(s)|(s-a)\dd s\\
&\leq \eta(t-a)\int_{a}^{t}(|\varphi_{j}'(s)-\varphi_{j}'(a)|+|\varphi_{j}'(a)|) \dd s\\
&\leq \eta(\eta+|\varphi_{j}'(a)|)(t-a)^{2}.
\end{align*}
For the second term we estimate as follows:
\begin{align*}
\int_{a}^{t} (s-a)|\varphi_{i}'(a)||\varphi_{j}'(s)-\varphi_{j}'(a)|\dd s &\leq |\varphi_{i}'(a)|(t-a)\int_{a}^{t} |\varphi_{j}'(s)-\varphi_{j}'(a)| \dd s\\
&\leq \eta |\varphi_{i}'(a)| (t-a)^{2}.
\end{align*}
Combining the estimates of each term gives
\[\int_{a}^{t} |\varphi_{i}(s)\varphi_{j}'(s)-(s-a)\varphi_{i}'(a)\varphi_{j}'(a)| \dd s \leq \eta(|\varphi_{i}'(a)|+|\varphi_{j}'(a)|+\eta)(t-a)^{2}.\]
In exactly the same way, we find:
\[\int_{a}^{t} |\varphi_{j}(s) \varphi_{i}'(s)-(s-a)\varphi_{i}'(a)\varphi_{j}'(a)| \dd s \leq \eta(|\varphi_{i}'(a)|+|\varphi_{j}'(a)|+\eta)(t-a)^{2}.\]
Using the triangle inequality then gives:
\[\int_{a}^{t} (\varphi_{i}\varphi_{j}'-\varphi_{j}\varphi_{i}') \leq 2\eta (|\varphi_{i}'(a)|+|\varphi_{j}'(a)|+\eta)(t-a)^2.\]
This proves (2).
\end{proof}

\subsection{Interpolation in the gaps}
Write $[m,M]\setminus K=\cup_{l \geq 1} (a_l,b_l)$ for some $a_l, b_l \in K$. Notice $(b_l - a_l)\to 0$. Consequently we may choose $\varepsilon_{l}\downarrow 0$ such that $(b_{l}-a_{l})<\min\{\varepsilon_{l}, \delta(\varepsilon_{l})\}$ for sufficiently large $l$. Using the definition of $\delta$, \eqref{lc2} and Lemma \ref{curvecontrol} for large $l$, and simply choosing $\varepsilon_{l}$ to be large for small $l$, we can ensure that for every $l\geq 1$, the following inequalities hold with $a=a_{l}, b=b_{l}$, $\varepsilon=\varepsilon_{l}$ and $\varphi$ dependent on $l$ through $a_{l}$ as before:
\begin{align}
(b-a)&\leq \varepsilon, \label{int1}\\
|\varphi_{i}'(b)-\varphi_{i}'(a)|&\leq \varepsilon &&\mbox{ if }1\leq i\leq r, \label{int2}\\
|\varphi_{i}(b)-(b-a)\varphi_{i}'(a)| &\leq \varepsilon (b-a) &&\mbox{ if }1\leq i\leq r, \label{int3}
\end{align}
and:
\begin{equation}
|\varphi_{ij}(b)|\leq \varepsilon( |\varphi_{i}'(a)| + |\varphi_{j}'(a)| + \varepsilon)(b-a)^{2} \qquad \mbox{ if }1\leq j<i\leq r. \label{int4}
\end{equation}

\begin{lemma}\label{interpolatingcurve}
Fix $l\geq 1$ and let $a=a_{l},\, b=b_{l}$ and $\varepsilon=\varepsilon_{l}$. Then there is a constant $C$ (independent of $l$) and a $C^1$ horizontal curve $\psi \colon [a,b] \to \mathbb{G}$ (depending on $l$) satisfying:
\begin{enumerate}
\item $\psi(a)=\gamma(a)$ and $\psi(b)=\gamma(b)$,
\item $\psi'(a)=\gamma'(a)$ and $\psi'(b)=\gamma'(b)$,
\item $|\psi'-\gamma'(a)|\leq C\varepsilon$ on $[a,b]$.
\end{enumerate}
\end{lemma}

\begin{claim}\label{switch}
Suppose the conclusion of Lemma \ref{interpolatingcurve} holds for some fixed $l$ with $\gamma$ replaced by $\varphi$. Then it holds for the same $l$ with the original curve $\gamma$.
\end{claim}

\begin{proof}
Choose $\psi$ satisfying the conclusions of Lemma \ref{interpolatingcurve} with $\gamma$ replaced by $\varphi$. Define a $C^{1}$ horizontal curve $\eta\colon \mathbb{G}\to \mathbb{G}$ by $\eta= L_{\gamma(a)}\circ F_{a}^{-1}\circ \psi$. Using the definition $\varphi=F_{a}\circ L_{\gamma(a)^{-1}} \circ \gamma$, it is straightforward to check that:
\[\eta(a)=\gamma(a)\mbox{ and } \eta(b)=\gamma(b),\]
\[\eta'(a)=\gamma'(a) \mbox{ and }\eta'(b)=\gamma'(b).\]
Recall that in the first $r$ coordinates translations are Euclidean and $F_{a}$ acts as a linear isometry. Hence:
\begin{equation}\label{initial}|\eta_{i}'-\gamma_{i}'(a)|\leq C\varepsilon \mbox{ on }[a,b] \mbox{ for }1\leq i\leq r.\end{equation}
It remains to check $|\eta_{ij}'-\gamma_{ij}'(a)|\leq C\varepsilon$ on $[a,b]$ for $i>j$. We estimate as follows:
\begin{align*}
|\eta_{ij}'-\gamma_{ij}'(a)|&= \frac{1}{2}| \eta_{i}\eta_{j}'-\eta_{j}\eta_{i}'-\gamma_{i}(a)\gamma_{j}'(a)+\gamma_{j}(a)\gamma_{i}'(a)|\\
&\leq \frac{1}{2}(|\eta_{i}\eta_{j}'-\gamma_{i}(a)\gamma_{j}'(a)| + |\eta_{j}\eta_{i}'-\gamma_{j}(a)\gamma_{i}'(a)|).
\end{align*}
We show how to estimate $|\eta_{i}\eta_{j}'-\gamma_{i}(a)\gamma_{j}'(a)|$; the argument for the other term is similar. Notice:
\[|\eta_{i}\eta_{j}'-\gamma_{i}(a)\gamma_{j}'(a)|\leq |\eta_{i}||\eta_{j}'-\gamma_{j}'(a)|+|\gamma_{j}'(a)||\eta_{i}-\gamma_{i}(a)|.\]
The term $|\eta_{i}|$ is bounded by a constant depending only on a compact set containing the image of $\gamma$, while \eqref{initial} states $|\eta_{j}'-\gamma_{j}'(a)|\leq C\varepsilon$. The term $|\gamma_{j}'(a)|$ is bounded by a constant independent of $l$ because $a\in K$ and $\gamma'$ is continuous on the compact set $K$. Combining this with \eqref{initial} shows that $\eta_{i}'$ is bounded by a constant independent of $l$. Using also $\gamma_{i}(a)=\eta_{i}(a)$ gives:
\begin{align*}
|\eta_{i}(t)-\gamma_{i}(a)|&=|\eta_{i}(t)-\eta_{i}(a)|\\
&\leq \int_{a}^{t} |\eta_{i}'|\\
&\leq C(t-a)\\
&\leq C\varepsilon.
\end{align*}
We deduce $|\eta_{ij}'-\gamma_{ij}'(a)|\leq C\varepsilon$ on $[a,b]$ if $i>j$, with $C$ independent of $l$, as required.
\end{proof}

We use the remainder of this subsection to prove Lemma \ref{interpolatingcurve} with $\varphi$ in place of $\gamma$, using the estimates \eqref{int1}--\eqref{int4}. To simplify notation, let $h=(b-a)/2$, $q=\varphi(b)$, $v=\varphi'(a)=(L,0,\ldots,0)$ and $w=\varphi'(b)$. Using \eqref{int1}--\eqref{int4} gives:
\begin{align}
& 0<h\leq \varepsilon, \label{not1}\\
& v=(L,0,\ldots, 0), \, w \mbox{ is horizontal at }q, \, |w_{i}-v_{i}|\leq \varepsilon \mbox{ if }1\leq i\leq r,\label{not2}\\
& |q_{1}-2Lh|\leq 2\varepsilon h \mbox{ and } |q_{i}|\leq 2\varepsilon h \mbox{ if } i>1, \label{not3}\\
& |q_{i1}| \leq 4\varepsilon (\varepsilon+L)h^{2} \mbox{ if }i>1 \mbox{ and } |q_{ij}|\leq 4\varepsilon^{2}h^{2} \mbox{ if }i>j>1. \label{not4}
\end{align}

By reparameterization, to prove Lemma \ref{interpolatingcurve} with $\varphi$ in place of $\gamma$, we will construct a suitable curve $\psi$ on $[0,2h]$ instead of $[a, a+2h]=[a,b]$. Using \eqref{not1}--\eqref{not4} one can forget completely about the original curves $\gamma$ and $\varphi$; we desire a $C^{1}$ horizontal curve $\psi\colon [0,2h]\to \mathbb{G}$ satisfying:
\begin{enumerate}
\item $\psi(0)=0$ and $\psi(2h)=q$,
\item $\psi'(0)=v$ and $\psi'(2h)=w$,
\item $|\psi'-v|\leq C\varepsilon$ on $[0,2h]$.
\end{enumerate}

To do this, we separately construct curves $\alpha$ on $[0,h]$ and $\beta$ on $[h, 2h]$, then define $\psi$ as the concatenation of $\alpha$ and $\beta$. Intuitively, the curve $\alpha$ on $[0,h]$ will fix the vertical component of the desired endpoint $q$, while the curve $\beta$ on $[h,2h]$ will fix the horizontal component of $q$ and the derivative $w$ at the final point.

We first construct the curve $\beta$ on $[h,2h]$. This will be relatively easy: the vertical coordinates of the start point of $\beta$ (which will be the endpoint of $\alpha$) have some flexibility, so we can define $\beta$ as the lift of a curve in $\mathbb{R}^r$ without too many constraints.

\begin{claim}\label{velocityfix}
There exists $\tilde{q}\in \mathbb{G}$ and a $C^1$ horizontal curve $\beta \colon [h, 2h]\to \mathbb{G}$ such that:
\begin{enumerate}
\item $\beta(h)=\tilde{q}$ and $\beta(2h)=q$,
\item $\beta'(h)=v$ and $\beta'(2h)=w$,
\item $\tilde{q}_{1}=Lh$, $\tilde{q}_{i}=0$ for $i>1$, and $v$ is horizontal at $\tilde{q}$,
\item $|\tilde{q}_{i1}|\leq C\varepsilon (L+\varepsilon) h^{2}$ for $i>1$ and $|\tilde{q}_{ij}|\leq C\varepsilon^{2}h^{2}$ for $i>j>1$,
\item $|\beta'-v|\leq C\varepsilon$.
\end{enumerate}
\end{claim}

\begin{proof}[Proof of Claim \ref{velocityfix}]
We start by defining a curve $c\colon [0,h]\to \mathbb{G}$ starting at $q$. The point $\tilde{q}$ will then be the endpoint of $c$ and the curve $\beta$ will be obtained by reversing the orientation of $c$ (and reparameterizing so the domain is $[h,2h]$). 

Denote $Q=(q_{1}, \ldots, q_{r})$, $W=-(w_{1}, \ldots, w_{r})$ and $V=-(v_{1}, \ldots, v_{r})$. Define a $C^1$ curve $\tilde{c}\colon [0,h]\to \mathbb{R}^{r}$ by the formula:
\[\tilde{c}(t)=Q+tW+\frac{t^2}{h^2}(-4hV-2hW-3Q)+\frac{t^3}{h^3}(3hV+hW+2Q).\]
Equations \eqref{not2} and \eqref{not3} imply $|V-W|\leq r\varepsilon$ and $|Q+2hV|\leq 2r\varepsilon h$. We claim that $\tilde{c}=(c_{1},\ldots, c_{r})$ has the following properties:
\begin{equation}\label{cbdry} \tilde{c}(0)=(q_{1},\ldots, q_{r}) \text{ and }\tilde{c}(h)=(Lh,0,\ldots, 0),\end{equation}
\begin{equation}\label{cprimebdry} \tilde{c}'(0)=-(w_{1}, \ldots, w_{r}) \text{ and }\tilde{c}'(h)=-(v_{1}, \ldots, v_{r}),\end{equation}
\begin{equation}\label{derivbound} |\tilde{c}'+(v_{1}, \dots, v_{r})|\leq C\varepsilon.\end{equation}

For \eqref{cbdry} the first equality $\tilde{c}(0)=Q$ is obvious. To prove the second, notice:
\begin{align*}
\tilde{c}(h)&=Q+hW-4hV-2hW-3Q+3hV+hW+2Q\\
&=-hV\\
&=(Lh,0,\ldots, 0).
\end{align*}

For \eqref{cprimebdry}, first notice:
\[\tilde{c}'(t)=W+\frac{2t}{h^2}(-4hV-2hW-3Q)+\frac{3t^2}{h^3}(3hV+hW+2Q).\]
Clearly $\tilde{c}'(0)=W=-(w_{1},\ldots, w_{r})$. For the second equality, notice:
\begin{align*}
\tilde{c}'(h)&=W+\frac{2}{h}(-4hV-2hW-3Q)+\frac{3}{h}(3hV+hW+2Q)\\
&= V\\
&=-(v_{1}, \ldots, v_{r}).
\end{align*}

For \eqref{derivbound}, use $|V-W|\leq r\varepsilon$ and $|Q+2hV|\leq 2r\varepsilon h$ to estimate as follows:
\begin{align*}
|\tilde{c}'(t)+(v_{1}, \ldots, v_{r})|&= |\tilde{c}'(t)-V|\\
&\leq |V-W| + \frac{2}{h}|-4hV-2hW-3Q|\\
&\qquad + \frac{3}{h}|3hV+hW+2Q|\\
&\leq C\varepsilon + \frac{2}{h}|-6hV-3Q| +\frac{3}{h}|4hV+2Q|\\
&\leq C\varepsilon.
\end{align*}
This proves \eqref{derivbound}.

In what follows we will need only properties \eqref{cbdry}--\eqref{derivbound} and not the precise definition of $\tilde{c}$. Let $c\colon [0,h] \to \mathbb{G}$ be the horizontal lift of $\tilde{c}$ to $\mathbb{G}$, starting at $q$. Hence for $i>j$ and $t\in [0,h]$:
\begin{equation}\label{verticalc}
c_{ij}(t)=q_{ij}+\frac{1}{2}\int_{0}^{t} (c_{i}c_{j}'-c_{j}c_{i}'),
\end{equation}
and consequently $c_{ij}'=(c_{i}c_{j}'-c_{j}c_{i}')/2$.

Clearly $c$ is a $C^{1}$ horizontal curve and $c(0)=q$. Using \eqref{verticalc} gives:
\begin{align*}
c_{ij}'(0)&=(1/2)(c_{i}(0)c_{j}'(0)-c_{j}(0)c_{i}'(0))\\
&=(1/2)(q_{i}(-w_{j})-q_{j}(-w_{i}))\\
&=-(1/2)(q_{i}w_{j}-q_{j}w_{i})\\
&=-w_{ij}
\end{align*}
since $w$ is horizontal at $q$. Since $c_{i}'(0)=-w_{i}$ for $1\leq i\leq r$, we deduce $c'(0)=-w$. A similar argument shows that $c'(h)=-v$. 

By construction, $|c_{i}'(t)+v_{i}|\leq C\varepsilon$ for $t\in [0,h]$. Since $v=(L,0,\ldots,0)$, we deduce $|c_{1}'|\leq L+C\varepsilon$ and $|c_{i}'|\leq C\varepsilon$ for $i>1$. Recall $c_{i}(0)=q_{i}$ and the estimates $|q_{1}|\leq 2Lh+2\varepsilon h$ and $|q_{i}|\leq 2\varepsilon h$ for $i>1$, which follow from \eqref{not3}. We obtain:
\[|c_{1}|\leq (2Lh+2\varepsilon h) +(L+C\varepsilon)h \leq C(L+\varepsilon)h,\]
\[|c_{i}|\leq 2\varepsilon h+C\varepsilon h\leq C\varepsilon h \mbox{ if } i>1.\]
We estimate the derivatives of the other coordinates of $c$ for $i>j$ as follows (using, in particular, $i>1$):
\begin{align*}
|c_{ij}'|&=\frac{1}{2}|c_{i}c_{j}'-c_{j}c_{i}'|\\
&\leq \frac{1}{2}(|c_{i}||c_{j}'|+|c_{j}||c_{i}'|)\\
&\leq \frac{1}{2}(C\varepsilon \cdot C + C\cdot C\varepsilon)\\
&\leq C\varepsilon.
\end{align*}
We deduce that $|c'+v|\leq C\varepsilon$.

Define $\tilde{q}=c(h)$. As $\tilde{c}(h)=(Lh,0,\ldots, 0)$, we observe $\tilde{q}_{1}=h$ and $\tilde{q}_{i}=0$ for $i>1$. Hence $v=LX_{1}(\tilde{q})$ so $v$ is horizontal at $\tilde{q}$. For $i>1$ we can use \eqref{not4} to estimate:
\begin{align*}
|\tilde{q}_{i1}|&= |c_{i1}(h)|\\
&= \left|q_{i1}+ \frac{1}{2}\int_{0}^{h} (c_{i}c_{1}'-c_{1}c_{i}')\right|\\
&\leq 4\varepsilon (L+\varepsilon) h^{2} + \frac{1}{2}\int_{0}^{h} |c_{i}||c_{1}'|+|c_{1}||c_{i}'|\\
&\leq 4\varepsilon (L+\varepsilon) h^{2}  + \frac{1}{2}\int_{0}^{h} C\varepsilon h \cdot (L+C\varepsilon) + C(L+\varepsilon)h \cdot C\varepsilon \\
&\leq C\varepsilon (L+\varepsilon) h^{2}.
\end{align*}
A similar argument shows that $|\tilde{q}_{ij}|\leq C\varepsilon^{2} h^{2}$ for $i>j>1$.

To summarise, the $C^1$ horizontal curve $c\colon [0,h]\to \mathbb{G}$ and $\tilde{q}\in \mathbb{G}$ satisfy:
\begin{enumerate}
\item $c(0)=q$ and $c(h)=\tilde{q}$,
\item $c'(0)=-w$, and $c'(h)=-v$,
\item $\tilde{q}_{1}=Lh$, $\tilde{q}_{i}=0$ for $i>1$, and $v$ is horizontal at $\tilde{q}$,
\item $|\tilde{q}_{i1}|\leq C\varepsilon (L+\varepsilon) h^{2}$ for $i>1$ and $|\tilde{q}_{ij}|\leq C\varepsilon^{2}h^{2}$ for $i>j>1$,
\item $|c'+v|\leq C\varepsilon$.
\end{enumerate}
Now define $\beta\colon [h,2h]\to \mathbb{G}$ by $\beta(t)=c(2h-t)$. Clearly $\beta$ is a $C^1$ horizontal curve with the properties given in the statement of the claim.
\end{proof}

Fix a curve $\beta$ and a point $\tilde{q}$ as given by Claim \ref{velocityfix}. We now construct the curve $\alpha$ on $[0,h]$. In order to reach the correct endpoint $\tilde{q}$, the curve $\alpha$ will be defined as a concatenation of lifts of curves in $\mathbb{R}^{r}$, each of which enclose prescribed signed areas in various planes, corresponding to the different vertical coordinates we wish to control.

\begin{claim}\label{interpolationsimplified}
There exists a $C^1$ horizontal curve $\alpha \colon [0,h] \to \mathbb{G}$ such that:
\begin{enumerate}
\item $\alpha(0)=0$ and $\alpha(h)=\tilde{q}$,
\item $\alpha'(0)=\alpha'(h)=v$,
\item $|\alpha'-v|\leq C\varepsilon$.
\end{enumerate}
\end{claim}

\begin{proof}[Proof of Claim \ref{interpolationsimplified}]

Divide $[0,h]$ into $N:=r(r-1)/2$ subintervals
\[I_k=[kh/N,(k+1)h/N], \qquad 0 \leq k \leq (N-1)\]
of length $h/N$. We will define a curve $\alpha$ so that, on each subinterval $I_k$, $\alpha$ fixes one of the $N=n-r$ vertical coordinates that it has to reach, without distorting the other coordinates.

There are two types of vertical coordinates to consider. The first type have the form $i1$ ($i>1$), while the second type have the form $ij$ ($i>j>1$). It will not matter in which order we fix the vertical coordinates, so arbitrarily assign to each interval $I_k$ a vertical coordinate $i1$ ($i>1$) or $ij$ ($i>j>1$) to be fixed. 

We define $\alpha$ on $[0,h]$ inductively. Let $\alpha(0)=0$. Fix $0 \leq k \leq (N-1)$ for which a $C^1$ horizontal curve $\alpha \colon [0,kh/N] \to \mathbb{G}$ has been defined satisfying:
\begin{enumerate}
\item $(\alpha_{1}(kh/N), \ldots , \alpha_{r}(kh/N))=(Lkh/N, 0, \ldots, 0)$,
\item $\alpha_{ij}(kh/N)=\tilde{q}_{ij}$ for the $k$ vertical coordinates $ij$ that have already been fixed,
\item $\alpha_{ij}(kh/N)=0$ whenever $ij$ is a vertical coordinate that remains to be fixed,
\item $\alpha'(0)=\alpha'(kh/N)=v$ if $k>0$,
\item $|\alpha'-v|\leq C\varepsilon$.
\end{enumerate}
We now show how to define $\alpha$ on $I_{k}$, with a different argument in each of three cases.

\begin{case}
Suppose the vertical coordinate to be fixed has the form $i1$ for $i>1$ and additionally $L\geq \varepsilon$ (the case $L<\varepsilon$ is treated separately below).

Fix $\lambda \in \mathbb{R}$ to be chosen later. Define $\tilde{g}\colon [0,h/N]\to \mathbb{R}^{r}$ by:
\[
\tilde{g}_{\mu}(t) = 
  \begin{cases} 
(Lkh/N)+Lt & \text{if } \mu=1, \\
\lambda h \varepsilon (1-\cos(2\pi N t/h))     & \text{if }\mu=i,\\
0 &\text{if }\mu \neq 1,i.
\end{cases}
\]
Notice
\[\tilde{g}(0)=(Lkh/N, 0, \ldots, 0)=(\alpha_{1}(kh/N), \ldots, \alpha_{r}(kh/N))\]
using the inductive hypothesis. 

Let $g\colon [0,h/N]\to \mathbb{G}$ be the horizontal lift of $\tilde{g}$ starting at $\alpha(kh/N)$. Note that the definition of horizontal lift implies
\begin{equation}\label{horizliftcase1}
g_{\mu\nu}(t)=\alpha_{\mu\nu} (kh/N) + \frac{1}{2}\int_{0}^{t} (g_{\mu}g_{\nu}'-g_{\nu}g_{\mu}')
\end{equation}
for every $\mu>\nu$ and $t\in [0,h/N]$. Clearly $g(0)=\alpha(kh/N)$ and
\[(g_{1}(h/N), \ldots, g_{r}(h/N))=(L(k+1)h/N,0,\ldots, 0).\]
If $(\mu,\nu)\neq (i,1)$ then either $g_{\mu}$ or $g_{\nu}$ is identically zero, which implies that $g_{\mu\nu}(h/N)=\alpha_{\mu\nu}(kh/N)$. For $(\mu,\nu)=(i,1)$ recall the assumption that $(i,1)$ remains to be fixed implies that $\alpha_{i1}(kh/N)=0$. Hence integration by parts shows:
\[g_{i1}(h/N)=\frac{1}{2}\int_{0}^{h/N} (g_{i}g_{1}' - g_{1}g_{i}') = \lambda Lh^2 \varepsilon / N.\]
Now choose $\lambda$ such that $\lambda Lh^2 \varepsilon / N=\tilde{q}_{i1}$, so that $g_{i1}(h/N)=\tilde{q}_{i1}$. The inequality $|\tilde{q}_{i1}|\leq C\varepsilon (L+\varepsilon) h^2$ implies that $\lambda \leq C(1+\varepsilon/L)$. Using the assumption $L\geq \varepsilon$ implies:
\begin{equation}\label{lambdabounded} \lambda \mbox{ is bounded by a fixed constant}.\end{equation}

We next claim that $g'(0)=g'(h/N)=(L,0,\ldots, 0)$. These equalities are obvious for the first $r$ coordinates of $g$, since they hold for $\tilde{g}$ and $g$ is the horizontal lift of $\tilde{g}$. For the vertical coordinates, notice that \eqref{horizliftcase1} implies $g_{\mu\nu}'=(g_{\mu}g_{\nu}'-g_{\nu}g_{\mu}')/2$. If $(\mu,\nu)\neq (i,1)$ then either $g_{\mu}$ or $g_{\nu}$ is identically zero, so $g_{\mu\nu}'$ is identically zero. If $(\mu,\nu)=(i,1)$ then the claim holds for the coordinate $i1$ because:
\[g_{i}(0)=g_{i}(h/N)=g_{i}'(0)=g_{i}'(h/N)=0.\]

Since $g_{1}'=L$ and $g_{\mu\nu}'$ is identically zero if $(\mu,\nu)\neq (i,1)$, the following inequality holds:
\[|g'-(L,0,\ldots, 0)| \leq |g_{i}'| + |g_{i1}'|.\]
Clearly
\[|g_{i}'|\leq \lambda h \varepsilon \cdot 2\pi N / h\leq C\varepsilon\]
since \eqref{lambdabounded} states that $\lambda$ is bounded by a constant. Recall that $\gamma'|_{K}$ is continuous so $L$ is bounded by a constant, as noted in \eqref{Lbounded}. For the second term, notice:
\[|g_{i1}'|=\frac{1}{2}|g_{i}g_{1}'-g_{1}g_{i}'|\leq \frac{1}{2}(2 \lambda h \varepsilon \cdot L + Lh \cdot C\varepsilon) \leq C\varepsilon,\]
since $\lambda$, $h$ and $L$ are bounded by constants. Hence $|g'-(L,0,\ldots, 0)| \leq C\varepsilon$.

Define $\alpha(t)=g(t - kh/N)$ for $t\in [kh/N,(k+1)h/N]$. The curve $\alpha$ now satisfies the inductive hypotheses with $k$ replaced by $k+1$.
\end{case}

\begin{case}
Suppose the vertical coordinate to be fixed has the form $i1$ for $i>1$ and instead $L< \varepsilon$. 

In this case $\alpha$ on $I_{k}$ will be defined by concatenating two curves $g$ and $h$ and reparameterizing. Assume that $\tilde{q}_{i1}>0$; a similar argument works in the other cases. We first define the curve $g$ which will be a circle traced out at variable speed; if $\tilde{q}_{i1}<0$ then one should choose a circle whose signed area has the opposite sign, while if $\tilde{q}_{i1}=0$ then one should choose the constant curve $g(t)=(Lkh/N,0,\ldots,0)$. Choose $\lambda=\sqrt{\tilde{q}_{i1}/\pi}>0$. Since $L<\varepsilon$, we have:
\[|\tilde{q}_{i1}| \leq C\varepsilon(L+\varepsilon)h^{2}\leq C\varepsilon^{2} h^{2}.\]
Hence:
\begin{equation}\label{lambdasmall} \lambda \leq C\varepsilon h.\end{equation}
Fix a $C^{1}$ map $\theta \colon [0,h/2N]\to 2\pi$ such that:
\begin{itemize}
\item $\theta(0)=0$ and $\theta(h/2N)=2\pi$,
\item $\theta'(0)=\theta'(h/2N)=L/\lambda$,
\item $|\theta'|\leq C\max \{L/\lambda, 1/h\}$.
\end{itemize}
Define a curve $\tilde{g} \colon [0,h/2N]\to \mathbb{R}^{r}$ by:
\[
\tilde{g}_{\mu}(t) = 
  \begin{cases} 
(Lkh/N)+\lambda \sin (\theta(t)) & \text{if } \mu=1, \\
\lambda \cos (\theta(t)) -\lambda     & \text{if }\mu=i,\\
0 &\text{if }\mu \neq 1,i.
\end{cases}
\]
Notice:
\[\tilde{g}(0)=\tilde{g}(h/2N)=(\alpha_{1}(kh/N), \ldots, \alpha_{r}(kh/N))\]
using the definition of $\tilde{g}$ and the inductive hypotheses. Let $g\colon [0,h/2N]\to \mathbb{G}$ be the horizontal lift of $\tilde{g}$ starting at $\alpha(kh/N)$, so:
\[g_{\mu\nu}(t)=\alpha_{\mu\nu} (kh/N) + \frac{1}{2}\int_{0}^{t} (g_{\mu}g_{\nu}'-g_{\nu}g_{\mu}')\]
for every $\mu>\nu$ and $t\in [0,h/2N]$.

Clearly $g(0)=\alpha(kh/N)$ and
\[(g_{1}(h/2N), \ldots, g_{r}(h/2N))=\tilde{g}(h/2N)=(\alpha_{1}(kh/N), \ldots, \alpha_{r}(kh/N)).\]
If $(\mu,\nu)\neq (i,1)$ then either $g_{\mu}$ or $g_{\nu}$ is identically zero, which implies $g_{\mu\nu}(h/2N)=\alpha_{\mu\nu}(kh/N)$. For $(\mu,\nu)=(i,1)$ recall the assumption that $(i,1)$ remains to be fixed implies that $\alpha_{i1}(kh/N)=0$. An elementary computation shows:
\[g_{i1}(h/2N)=\frac{1}{2}\int_{0}^{h/2N} (g_{i}g_{1}'-g_{1}g_{i}')=\pi \lambda^{2}.\]
Recalling the definition of $\lambda$, we deduce $g_{i1}(h/2N)=\tilde{q}_{i1}$.

We next claim that $g'(0)=g'(h/2N)=(L,0,\ldots, 0)$. These equalities are elementary for the first $r$ coordinates of $g$, since they hold for $\tilde{g}$ and $g$ is the horizontal lift of $\tilde{g}$. For the vertical coordinates, recall the equality $g_{\mu\nu}'=(g_{\mu}g_{\nu}'-g_{\nu}g_{\mu}')/2$. If $(\mu,\nu)\neq (i,1)$ then $g_{\mu}$ or $g_{\nu}$ is identically zero, so $g_{\mu\nu}'$ is identically zero. If $(\mu,\nu)=(i,1)$ then the claim holds for the coordinate $i1$ because:
\[g_{i}(0)=g_{i}(h/2N)=g_{i}'(0)=g_{i}'(h/2N)=0.\]

Since $g_{\mu\nu}'$ is identically zero if $(\mu,\nu)\neq (i,1)$, we have:
\[|g'-(L,0,\ldots, 0)|\leq |g_{1}'-L|+|g_{i}'|+|g_{i1}'|\leq L+|g_{1}'|+|g_{i}'|+|g_{i1}'|.\]
To bound this, we first recall the assumption $L<\varepsilon$. Next it is easy to see $|g_{1}'|, |g_{i}'|\leq \lambda |\theta'|$. Since \eqref{lambdasmall} states that $\lambda \leq C\varepsilon h$, we have:
\[\lambda |\theta'| \leq C \max \{L, \lambda/h\}\leq C\varepsilon.\]
Consequently $|g_{1}'|, |g_{i}'|\leq C\varepsilon$. For the final term we again use \eqref{lambdasmall} to obtain:
\begin{align*}
|g_{i1}'|&=\frac{1}{2}|g_{i}g_{1}'-g_{1}g_{i}'|\\
&\leq \frac{1}{2} (2\lambda \cdot C\varepsilon + (Lh+\lambda)\cdot C\varepsilon) \\
&\leq C\varepsilon.
\end{align*}
Hence 
\[|g'-(L,0,\ldots, 0)|\leq C\varepsilon.\]

To summarize, $g\colon [0,h/2N]\to \mathbb{G}$ is a $C^{1}$ horizontal curve with the following properties:
\begin{itemize}
\item $g(0)=\alpha(kh/N)$,
\item $(g_{1}(h/2N), \ldots, g_{r}(h/2N))=(\alpha_{1}(kh/N), \ldots, \alpha_{r}(kh/N))$,
\item $g_{\mu\nu}(h/2N)=\alpha_{\mu\nu}(kh/N)$ if $(\mu,\nu)\neq (i,1)$,
\item $g_{i1}(h/2N)=\tilde{q}_{i1}$,
\item $g'(0)=g'(h/2N)=(L,0,\ldots, 0)$,
\item $|g'-(L,0,\ldots, 0)|\leq C\varepsilon$.
\end{itemize}

Define a $C^{1}$ curve $\rho \colon [0,h/2N]\to \mathbb{G}$ by $\rho(t)=g(h/2N)+(2Lt,0,\ldots, 0)$. Since
\[(g_{1}(h/2N), \ldots, g_{r}(h/2N))=(kh/N, 0, \ldots, 0),\]
we see $\rho'(t)=(2L,0,\ldots, 0)=2LX_{1}(\rho(t))$, so $\rho$ is a horizontal curve. Clearly $\rho(0)=g(h/2N)$ and $\rho(h/2N)=g(h/2N)+(Lh/N,0,\ldots, 0)$, in particular 
\[(\rho_{1}(h/2N), \ldots, \rho_{r}(h/2N))=(L(k+1)h/N, 0, \ldots, 0).\]
Also notice $|\rho'-(L,0,\ldots, 0)|\leq L<\varepsilon$.

Define 
\[
 \alpha(t) = 
  \begin{cases} 
   g(t-kh/N) & \text{if } t\in [kh/N, (k+1/2)h/N], \\
   \rho(t-((k+1/2)h/N))       & \text{if } t\in [(k+1/2)h/N, (k+1)h/N].
  \end{cases}
\]
The curve $\alpha$ now satisfies the inductive hypotheses with $k$ replaced by $k+1$.
\end{case}

\begin{case}
Suppose the vertical coordinate to be fixed has the form $ij$ for $i>j>1$.

Fix $\lambda \in \mathbb{R}$ to be chosen later and let:
\[\zeta(t)=\lambda h \varepsilon (1-\cos(4\pi Nt/h)).\]
Suppose $\tilde{q}_{ij}\geq 0$; it will be clear that a similar argument works for $\tilde{q}_{ij}< 0$. Define $\tilde{g}\colon [0,h/N]\to \mathbb{R}^{r}$ by:
\[
\tilde{g}_{\mu}(t) = 
  \begin{cases} 
(Lkh/N)+Lt & \text{if } \mu=1, \\
\zeta(t) \cos(2\pi Nt/h)       & \text{if } \mu=i,\\
\zeta(t) \sin (2\pi Nt/h)	& \text{if } \mu=j,\\
0 & \text{if } \mu\neq 1,i,j.
  \end{cases}
\]
For the case $\tilde{q}_{ij}<0$, one should swap the formulae for $g_{i}$ and $g_{j}$. Notice
\[\tilde{g}(0)=(Lkh/N, 0, \ldots, 0)=(\alpha_{1}(kh/N), \ldots, \alpha_{r}(kh/N))\]
using the inductive hypothesis. Let $g\colon [0,h/N]\to \mathbb{G}$ be the horizontal lift of $\tilde{g}$ starting at $\alpha(kh/N)$.

Clearly:
\[\zeta(0)=\zeta'(0)=\zeta(h/N)=\zeta'(h/N)=0.\]
This implies that $g(0)=0$, $g'(0)=g'(h/N)=(L,0,\ldots, 0)$, and 
\[(g_{1}(h/N), \ldots, g_{r}(h/N))=(L(k+1)h/N,0,\ldots, 0).\]
Suppose $1\leq \nu<\mu\leq r$. If $\mu\neq 1,i,j$ or $\nu \neq1,i,j$ then $g_{\mu\nu}$ is identically zero, in particular $g_{\mu\nu}(h/N)=0$. To compute $g(h/N)$, it remains to calculate $g_{i1}(h/N), g_{j1}(h/N)$ and $g_{ij}(h/N)$.

First notice:
\begin{align*}
g_{i1}(h/N)&=\alpha_{i1}(kh/N)+\frac{1}{2}\int_{0}^{h/N} (g_{i}g_{1}'-g_{1}g_{i}')\\
&=\alpha_{i1}(kh/N)+\frac{1}{2}L\int_{0}^{h/N}g_{i}(t)-tg_{i}'(t)\dd t\\
&=\alpha_{i1}(kh/N)+L\int_{0}^{h/N}g_{i},
\end{align*}
using integration by parts and $g_{i}(h/N)=0$ to obtain the final equality. However, $\int_{0}^{h/N}g_{i}=0$ using orthogonality of the trigonometric system. Hence $g_{i1}(h/N)=\alpha_{i1}(kh/N)$. A similar argument gives the equality $g_{j1}(h/N)=\alpha_{j1}(kh/N)$.

It remains to calculate $g_{ij}(h/N)$. Since the coordinate $ij$ has not yet been fixed, $\alpha_{ij}(kh/N)=0$. Using elementary calculus, we obtain:
\begin{align*}
g_{ij}(h/N)&=\frac{1}{2}\int_{0}^{h/N} (g_{i}g_{j}'-g_{j}g_{i}')\\
&=\frac{2\pi N}{h}\int_{0}^{h/N} \zeta^2\\
&=\lambda^{2}h^{2}\varepsilon^{2} I,
\end{align*}
where
\[I=\int_{0}^{2\pi} (1-\cos 2s)^{2} \dd s,\]
which is a positive constant. Since we assumed that $\tilde{q}_{ij}\geq 0$, we may choose $\lambda\geq 0$ such that $\lambda^{2}h^{2}\varepsilon^{2} I=\tilde{q}_{ij}$. Using the inequality $|\tilde{q}_{ij}|\leq C\varepsilon^{2} h^{2}$ then implies:
\begin{equation}\label{lambdabounded2} \lambda \mbox{ is bounded by a constant}.\end{equation}
In the case $\tilde{q}_{ij}<0$, swapping the formulae for $g_{i}$ and $g_{j}$ gives $g_{ij}(h/N)=-\lambda^{2}h^{2}\varepsilon^{2} I$ and a similar argument applies.

Since $g_{\mu\nu}'$ is identically zero if $\mu \neq 1,i,j$ or $\nu \neq 1,i,j$, the following inequality holds:
\[|g'-(L,0,\ldots, 0)| \leq |g_{i}'| + |g_{j}'| + |g_{i1}'|+ |g_{j1}'| + |g_{ij}'|.\]
Notice, using \eqref{lambdabounded2}, $|\zeta|\leq 2\lambda h\varepsilon \leq Ch\varepsilon$ and $|\zeta'|\leq C\lambda \varepsilon \leq C\varepsilon$. Since
\[g_{i}'(t)=\zeta'(t)\cos(2\pi Nt/h) -\frac{2\pi N}{h}\zeta \sin(2\pi Nt/h),\]
we can estimate:
\[|g_{i}'|\leq |\zeta'| + \frac{2\pi N}{h}|\zeta|\leq C\varepsilon.\]
A similar argument bounds $|g_{j}'|$. Recall again that \eqref{Lbounded} states $L$ is bounded since $\gamma'|_{K}$ is continuous. Next, we have:
\begin{align*}
|g_{i1}'|=\frac{1}{2}|g_{i}g_{1}'-g_{1}g_{i}'| &\leq \frac{1}{2}(|\zeta|\cdot L + Lh\cdot C\varepsilon)\\
&\leq Ch\varepsilon\\
&\leq C\varepsilon.
\end{align*}
Similarly $|g_{j1}'|\leq C\varepsilon$. For the final term, notice
\begin{align*}
|g_{ij}'|=\frac{1}{2}|g_{i}g_{j}'-g_{j}g_{i}'| &\leq \frac{1}{2}(|\zeta|\cdot C\varepsilon + |\zeta|\cdot C\varepsilon)\\
&\leq Ch\varepsilon^2\\ 
&\leq C\varepsilon.
\end{align*}
Combining the preceeding estimates gives $|g'-(L,0,\ldots, 0)| \leq C\varepsilon$.

Define $\alpha(t)=g(t - kh/N)$ for $t\in [kh/N,(k+1)h/N]$. The curve $\alpha$ now satisfies the inductive hypotheses with $k$ replaced by $k+1$.
\end{case}

By repeatedly applying the inductive step, we obtain a $C^1$ horizontal curve $\alpha\colon [0,h]\to \mathbb{G}$ such that:
\begin{itemize}
\item $\alpha(0)=0$ and $\alpha(h)=\tilde{q}$, since $(\tilde{q}_{1}, \ldots, \tilde{q}_{r})=(Lh, 0, \ldots, 0)$ and every vertical coordinate was fixed at some stage in the construction of $\alpha$,
\item $\alpha'(0)=\alpha'(h)=(L,0,\ldots, 0)$,
\item $|\alpha'(t)-(L,0,\ldots, 0)|\leq C\varepsilon$ for all $t\in [0,h]$.
\end{itemize}

This concludes the proof of Claim \ref{interpolationsimplified}.
\end{proof}

\begin{proof}[Proof of Lemma \ref{interpolatingcurve}]
Concatenating the curves $\alpha$ and $\beta$ from Claim \ref{interpolationsimplified} and Claim \ref{velocityfix}, then reparameterizing so the domain is $[a,b]=[a,a+2h]$ instead of the interval $[0,2h]=[0,b-a]$, gives a curve $\psi$ with the properties desired in Lemma \ref{interpolatingcurve} with $\varphi$ in place of $\gamma$. Using Claim \ref{switch} then gives Lemma \ref{interpolatingcurve} for the original curve $\gamma$, as desired.
\end{proof}

\subsection{A smooth horizontal curve}

Recall $[m,M]\setminus K=\cup_{l\geq 1}(a_{l}, b_{l})$. Use Lemma \ref{interpolatingcurve} to choose, for each $l\geq 1$, $C^1$ horizontal curves $\psi_{l} \colon [a_{l},b_{l}] \to \mathbb{G}$ satisfying:
\begin{equation*}\psi_{l}(a_{l})=\gamma(a_{l}) \mbox{ and }\psi_{l}(b_{l})=\gamma(b_{l}),\end{equation*}
\begin{equation*}\psi_{l}'(a_{l})=\gamma'(a_{l}) \mbox{ and }\psi_{l}'(b_{l})=\gamma'(b_{l}),\end{equation*}
\begin{equation}\label{directions}|\psi_{l}'-\gamma'(a_{l})|\leq C\varepsilon_{l} \mbox{ on }[a_{l}, b_{l}].\end{equation}
We recall that $\varepsilon_{l}\downarrow 0$ as $l\to \infty$. Define a curve $\Gamma\colon [m,M]\to \mathbb{G}$ by:
\[ \Gamma(t) = \begin{cases} \gamma(t) & \mbox{if }t\in K,\\
\psi_{l}(t) & \mbox{if }t\in (a_{l},b_{l}) \mbox{ and } l\geq 1. \end{cases}\]

We now show that $\Gamma$ is the desired $C^{1}$ horizontal curve.

\begin{proposition}\label{smoothcurve}
The function $\Gamma\colon [m,M]\to \mathbb{G}$ defines a $C^{1}$ horizontal curve in $\mathbb{G}$ with derivative:
\[ \Gamma'(t) = \begin{cases} \gamma'(t) & \mbox{if }t\in K,\\
\psi_{l}'(t) & \mbox{if }t\in (a_{l},b_{l}) \mbox{ and } l\geq 1. \end{cases}\]
\end{proposition}

\begin{proof}
Clearly $\Gamma$ is $C^{1}$ and $\Gamma'=\psi_{l}'$ on each interval $(a_{l},b_{l})$, $l\geq 1$. 

We next check that $\Gamma'(t)$ exists and is equal to $\gamma'(t)$ for $t\in K$. Fix $t\in K$. We first check differentiability at $t$ from the right. If $t=a_{l}$ for some $l$ then differentiability on the right of $\Gamma$ at $t$ follows from that of $\psi_{l}$ and the equalities $\psi_{l}(a_{l})=\gamma(a_{l})$ and $\psi_{l}'(a_{l})=\gamma'(a_{l})$. Suppose $t\neq a_{l}$ for any $l$ and let $\delta>0$. If $t+\delta\in K$ then
\[|\Gamma(t+\delta)-\Gamma(t)-\delta\gamma'(t)|=|\gamma(t+\delta)-\gamma(t)-\delta \gamma'(t)|,\]
which is small relative to $\delta$, for $\delta$ sufficiently small, by differentiability of $\gamma$ at $t$. Suppose $t+\delta\notin K$. Then $t< a_{l}<t+\delta<b_{l}$ for some $l\geq 1$; by choosing $\delta$ small we can ensure that $l$ is large and hence $\varepsilon_{l}$ is small. In this case we estimate as follows:
\begin{align*}
|\Gamma(t+\delta)-\Gamma(t)-\delta\gamma'(t)|&\leq |\Gamma(t+\delta)-\Gamma(a_{l})-(t+\delta-a_{l}) \gamma'(t)|\\
&\qquad + |\Gamma(a_{l})-\Gamma(t)-(a_{l}-t)\gamma'(t)|\\
&\leq |\psi_{l}(t+\delta)-\psi_{l}(a_{l})-(t+\delta-a_{l})\gamma'(a_{l})|\\
&\qquad + |t+\delta-a_{l}||\gamma'(a_{l})-\gamma'(t)|\\
&\qquad \qquad + |\gamma(a_{l})-\gamma(t)-(a_{l}-t)\gamma'(t)|.
\end{align*}
The second term is small, relative to $\delta$, by using continuity of $\gamma'|_{K}$ and the simple estimate $|t+\delta-a_{l}|\leq t+\delta-t=\delta$. The third term is small, relative to $\delta$, by differentiability of $\gamma$ at $t$. For the first term we estimate further, using \eqref{directions},
\begin{align*}
|\psi_{l}(t+\delta)-\psi_{l}(a_{l})-(t+\delta-a_{l})\gamma'(a_{l})| &= \left| \int_{a_{l}}^{t+\delta} (\psi_{l}'(s)-\gamma'(a_{l}))\dd s\right|\\
&\leq \int_{a_{l}}^{t+\delta} C\varepsilon_{l} \dd s\\
&\leq C \varepsilon_{l} \delta.
\end{align*}
As indicated above, by choosing $\delta>0$ sufficiently small we can ensure that $\varepsilon_{l}$ is also small. This gives the estimate of the final term. Hence $\Gamma$ is differentiable at $t$ from the right with the required derivative $\gamma'(t)$. A similar argument, with $t+\delta$ replaced by $t-\delta$, gives differentiability of $\Gamma$ from the left. Hence $\Gamma$ is differentiable at $t\in K$ with $\Gamma'(t)=\gamma'(t)$, as claimed.

Clearly $\Gamma$ is everywhere horizontal, since $\gamma$ is horizontal in $K$ and each map $\psi_{l}$ is horizontal in $(a_{l},b_{l})$. We claim that $\Gamma$ is also $C^{1}$. Continuity of $\Gamma'$ in intervals $(a_{l},b_{l})$ is clear because each map $\psi_{l}$ is $C^{1}$. To conclude the proof, we must check continuity of $\Gamma'$ at points of $K$. Fix $t\in K$. We first show that $\Gamma'$ is continuous at $t$ from the right. We showed above that $\Gamma'(t)=\gamma'(t)$. If $t=a_{l}$ for some $l$ then continuity of $\Gamma'$ at $t$ follows from continuity of $\psi_{l}'$ and the equality $\psi_{l}'(a_{l})=\gamma'(a_{l})$. Suppose $t\neq a_{l}$ for any $l$ and let $\delta>0$. If $t+\delta\in K$ then $\Gamma'(t+\delta)=\gamma'(t+\delta)$, which is close to $\gamma'(t)$ for sufficiently small $\delta$ using continuity of $\gamma'|_{K}$. Suppose now that $t+\delta \notin K$. Then we can find $l\geq 1$ such that $t < a_{l}<t+\delta<b_{l}$. As before, if $\delta$ is sufficiently small then we can ensure that $\varepsilon_{l}$ is small. In this case we have:
\[|\Gamma'(t+\delta)-\Gamma'(t)|\leq |\psi_{l}'(t+\delta)-\gamma'(a_{l})|+|\gamma'(a_{l})-\gamma'(t)|.\]
This is small, for sufficiently small $\delta>0$, by \eqref{directions} for the first term and continuity of $\gamma'|_{K}$ for the second. This shows that $\Gamma'$ is continuous on the right. A similar argument gives continuity of $\Gamma'$ at $t$ on the left. This concludes the proof.
\end{proof}

Using the definition of $\Gamma$ and Proposition \ref{smoothcurve} gives:
\begin{align*}
\mathcal{L}^{1}\{t\in [m,M]:\Gamma(t)\neq \gamma(t)\mbox{ or }\Gamma'(t)\neq \gamma'(t)\} &\leq \mathcal{L}^{1}([m,M]\setminus K)\\
&<\varepsilon.
\end{align*}
Proposition \ref{smoothcurve} also states that $\Gamma$ is a $C^{1}$ horizontal curve on $[m,M]$. Extending $\Gamma$ arbitrarily to a $C^{1}$ horizontal curve on $[0,1]$ proves Theorem \ref{freelusin}. 

\section{Lusin approximation in general Carnot groups of step 2}\label{generalcase}

In this section we prove Theorem \ref{lusinstep2}, which extends Lusin approximation for horizontal curves to general Carnot groups of step 2. We first show how the approximation is preserved by suitable maps between Carnot groups.

\begin{theorem}\label{imagespreserve}
Let $\mathbb{G}$ and $\mathbb{H}$ be Carnot groups whose Lie algebras have horizontal layer $V$ and $W$, respectively. Suppose $\mathbb{G}$ admits Lusin approximation for horizontal curves and there is a Lie group homomorphism $F\colon \mathbb{G}\to \mathbb{H}$ such that $dF(V)=W$. Then $\mathbb{H}$ admits Lusin approximation for horizontal curves.
\end{theorem}

\begin{proof}
Let $\gamma \colon [0,1]\to \mathbb{H}$ be an absolutely continuous horizontal curve in $\mathbb{H}$. Fix $\varepsilon > 0$. Since group translations preserve $C^1$ horizontal curves, we can assume $\gamma(0)=0$ to find a $C^1$ horizontal approximation for $\gamma$. Fix a basis $Y_{1}, \ldots, Y_{r}$ of $W$. Since $\gamma$ is a horizontal curve in $\mathbb{H}$, there exist integrable functions $a_{i}\colon [0,1]\to \mathbb{R}$, $1\leq i\leq r$, such that for almost every $t\in [0,1]$:
\[\gamma'(t)=\sum_{i=1}^{r} a_{i}(t)Y_{i}(\gamma(t)).\]
Since $dF(V)=W$, we may find $X_{1}, \ldots, X_{r}\in V$ such that $dF(X_{i})=Y_{i}$ for $1\leq i\leq r$. It is well known that there exists an absolutely continuous curve $\alpha \colon [0,1]\to \mathbb{G}$ such that $\alpha(0)=0$ and for almost every $t\in [0,1]$:
\[\alpha'(t)=\sum_{i=1}^{r} a_{i}(t)X_{i}(\alpha(t)).\]
Since $X_{1}, \ldots, X_{r}\in V$, such a curve $\alpha$ is horizontal. Notice
\[F(\alpha(0))=F(0)=0=\gamma(0).\]
The chain rule implies that for almost every $t$:
\begin{align*}
(F\circ \alpha)'(t)&=dF_{\alpha(t)}(\alpha'(t))\\
&=dF_{\alpha(t)} \left(\sum_{i=1}^{r} a_{i}(t)X_{i}(\alpha(t)) \right)\\
&=\sum_{i=1}^{r}a_{i}(t)Y_{i}(F(\alpha(t))).
\end{align*}
Since $F\circ \alpha$ and $\gamma$ follow the same flow starting at 0, we deduce $F\circ \alpha=\gamma$.

Since Lusin approximation for horizontal curves holds in $\mathbb{G}$, we may choose a $C^1$ horizontal curve $\beta\colon [0,1]\to \mathbb{G}$ such that:
\[\mathcal{L}^{1} \{t\in [0,1] \colon \beta(t)\neq \alpha(t) \mbox{ or }\beta'(t)\neq \alpha'(t)\}<\varepsilon.\]
Define $\Gamma\colon [0,1] \to \mathbb{H}$ by $\Gamma = F\circ \beta$. Clearly $\Gamma$ is $C^1$ and horizontal, since $F$ is smooth and preserves the horizontal layer. To conclude the proof, we claim that if $t$ satisfies $\beta(t)=\alpha(t)$ and $\beta'(t)=\alpha'(t)$, then $\Gamma(t)=\gamma(t)$ and $\Gamma'(t)=\gamma'(t)$. Fix such a $t$. Then:
\[\Gamma(t)=F(\beta(t))=F(\alpha(t))=\gamma(t)\]
and
\begin{align*}
\Gamma'(t)=(F\circ \beta)'(t)&=dF_{\beta(t)}(\beta'(t))\\
&=dF_{\alpha(t)}(\alpha'(t))\\
&=(F\circ \alpha)'(t)\\
&=\gamma'(t).
\end{align*}
\end{proof}

We next use the definition of free Lie groups in terms of free-nilpotent Lie algebras (Definition \ref{freecarnotgroup}) to show that any Carnot group is an image of a free Carnot group of the same step by a map of the type in Theorem \ref{imagespreserve}.

\begin{theorem}\label{goodmapsexist}
Suppose $\mathbb{H}$ is a Carnot group of step $s$ whose Lie algebra has horizontal layer $W$. Then there is a free Carnot group $\mathbb{G}$ of step $s$, whose Lie algebra has horizontal layer $V$, and a Lie group homomorphism $F\colon \mathbb{G}\to \mathbb{H}$ such that $dF(V)=W$.
\end{theorem}

\begin{proof}
Let $Y_{1}, \ldots, Y_{r}$ be a basis of $W$. Let $\mathbb{G}$ be a free Carnot group of step $s$ whose Lie algebra is generated by a basis $X_{1}, \ldots, X_{r}$ of its horizontal layer $V$. Define $\Phi\colon \{X_{1}, \ldots, X_{r}\}\to \{Y_{1}, \ldots, Y_{r}\}$ by $\Phi(X_{i})=Y_{i}$ for $1\leq i\leq r$. Since the Lie algebra of $\mathbb{G}$ is free, $\Phi$ extends to a homomorphism $\tilde{\Phi}$ from the Lie algebra of $\mathbb{G}$ to the Lie algebra of $\mathbb{H}$. In particular, $\tilde{\Phi}(X_{i})=Y_{i}$ for $1\leq i\leq r$, which implies $\tilde{\Phi}(V)=W$. Since Carnot groups are simply connected Lie groups we can lift $\tilde{\Phi}$ to a Lie group homomorphism $F\colon \mathbb{G}\to \mathbb{H}$ such that $dF=\tilde{\Phi}$, in particular $dF(V)=W$.
\end{proof}

By Theorem \ref{freelusin} we know that free Carnot groups of step 2 admit Lusin approximation for horizontal curves. Theorem \ref{goodmapsexist} and Theorem \ref{imagespreserve} show that any Carnot group of step 2 is an image of a free Carnot group of step 2, under a map that preserves Lusin approximation for horizontal curves. We conclude that $C^1$ approximation for horizontal curves holds in any Carnot group of step 2, proving Theorem \ref{lusinstep2}.

\end{document}